\documentclass[review]{elsarticle}
\usepackage{lineno,hyperref}
\usepackage{amssymb}
\usepackage{amsmath}
\usepackage{amsthm}
\usepackage{dcolumn}
\usepackage{endnotes}
\usepackage{tabularx}
\usepackage{comment}
\usepackage[matrix,arrow]{xy}
\usepackage{color}

\newtheorem{theorem}{Theorem}[section]
\newtheorem{proposition}[theorem]{Proposition}
\newtheorem{lemma}[theorem]{Lemma}
\newtheorem{corollary}[theorem]{Corollary}

\theoremstyle{definition}
\newtheorem{definition}[theorem]{Definition}
\newtheorem{example}[theorem]{Example}
\newtheorem{remark}[theorem]{Remark}

\newcommand{\ua}{\mathord{\uparrow}}
\newcommand{\da}{\mathord{\downarrow}}

\newcommand{\cl}{{\rm cl}}

\newcommand{\I}{\item}
\newcommand{\II}{\begin{enumerate}}
\newcommand{\III}{\end{enumerate}}

\begin{document}

\begin{frontmatter}

\title{The reflectivity  of  some categories of  $T_0$ spaces  in domain theory}
\tnotetext[mytitlenote]{This work was supported by NSFC (1210010153, 12071188) and NIE ACRF (RI 3/16 ZDS), Singapore} 

\author{Chong Shen}
\address{School of Mathematical Sciences, Nanjing Normal University, Nanjing, Jiangsu,  China}
\corref{mycorrespondingauthor}
\cortext[mycorrespondingauthor]{Corresponding author}
\ead{shenchong0520@163.com}

\author{Xiaoyong Xi}
\address{School of Mathematics and Statistics, Jiangsu Normal University, Jiangsu, Xuzhou, China}

\author{Dongsheng Zhao}
\address{Mathematics and Mathematical Education, National Institute of Education,
	Nanyang Technological University,  1 Nanyang Walk, Singapore}

\begin{abstract}
Keimel and Lawson  proposed a set of  conditions for proving a category of  topological spaces to be reflective in the category of all $T_0$ spaces. These conditions were  recently  used to prove the  reflectivity of the category of all well-filtered spaces.  In this paper, we prove that, in certain sense,  these conditions are not just sufficient but also necessary for a category of $T_0$ spaces to be reflective. Using this general result, we easily deduce that several categories proposed  in domain theory are not reflective, thus answered  a few open problems.
 \end{abstract}


\end{frontmatter}


Given a full subcategory $\mathbf{D}$ of a category $\mathbf{C}$, one natural and frequently asked question is whether $\mathbf{D}$ is reflective in $\mathbf{C}$. The objects in $\mathbf{D}$ can be viewed as  ``special objects", the reflectivity of $\mathbf{D}$ ensures that every general object in $\mathbf{C}$ can be ``completed" to be a special object, or ``densely embedded into"  a special object.

 Keimel and Lawson \cite{KeimelLawson} showed that  a category ${\bf K}$ of  $T_0$ spaces is reflective in the category  ${\bf Top_0}$ of all $T_0$ spaces  if it satisfies the following {\bf four conditions}:
\II
\I[(K1)] ${\bf K}$ contains all sober spaces;
\I[(K2)] If $X\in{\bf K}$ and $Y$ is homeomorphic to $X$, then $Y\in{\bf K}$;
\I[(K3)] If $\{X_i:i\in I\}\subseteq{\bf K}$ is a family of subspaces of a sober space, then the subspace $\bigcap_{i\in I}X_i\in{\bf K}$.
\I[(K4)]  If $f:X\longrightarrow Y$ is a continuous mapping from a sober space $X$ to a sober space $Y$, then for any subspace $Y_1$ of $Y$, $Y_1\in{\bf  K}$ implies that $f^{-1}(Y_1)\in {\bf K}$.
\III

It is now known  that the category  of $d$-spaces, the category of well-filtered spaces and the category of sober spaces  all satisfy the above four conditions (see \cite{KeimelLawson, wuxixuzhao,Wyler}), hence they are reflective.

These four conditions can, however,  usually only be used to confirm the reflectivity of subcategories of ${\bf Top_0}$. In this paper we shall prove that, in certain sense,  they are also necessary  conditions, hence can be used to disprove the reflectivity of some subcategories of   ${\bf Top_0}$. In particular, we shall apply the main result to show that the categories of all co-sober spaces, the category of all open-well-filtered spaces,  the category of all strongly $k$-bounded sober spaces and the category of  all
strongly $d$-spaces are not reflective, thus give the answers to  a group of open problems.

For a subcategory ${\bf K}$ of ${\bf Top_0}$, we say that {\bf K} 
\II
\I[(1)] is \emph{productive}, if the product $\prod_{i\in I}X_i\in{\bf K}$ whenever $\{X_i:i\in I\}\subseteq{\bf K}$, and
\I[(2)] is \emph{$b$-closed-hereditary}, if $A\in{\bf K}$ whenever $A$ is a $b$-closed subspace of some $X\in{\bf K}$. 
\III

The following is our main result.

\begin{theorem}
For a full subcategory ${\bf K}$ of ${\bf Top_0}$ with ${\bf K}\nsubseteq {\bf Top_1}$, if ${\bf K}$ satisfies (K2), then the following four statements are equivalent:
\II	
\I[$(1)$]	${\bf K}$ is reflective in ${\bf Top_0}$;	
\I[$(2)$]	 ${\bf K}$ satisfies conditions (K1)--(K4);	
\I[$(3)$]	 ${\bf K}$ is productive and $b$-closed-hereditary;
\I[$(4)$]	${\bf K}$ is productive and has equalizers.
\III
\end{theorem}

\section{The $b$-topology and $b$-dense embedding}

Let $P$ a poset. For $A\subseteq P$, let $\da A=\{x\in P: x\leq a\text{ for some }a\in A\}$ and $\ua A=\{x\in P: x\geq a\text{ for some }a\in A\}$. For $x\in P$, we write $\da x$ for $\da\{x\}$ and $\ua x$ for $\ua \{x\}$, respectively.
A subset $A$ is called a \emph{lower set} (resp. \emph{upper set}) if $A=\da A$ (resp. $A=\ua A$).

For a $T_0$ space $X$,
the specialization order $\leq$ on $X$ is defined as $x\leq y$ iff $x\in \cl(y)$, where $\cl$ is
the  closure operator of $X$. In the following, when we consider a $T_0$ space $X$ as a poset, it is always equipped with  the specialization order.

For a $T_0$ space $X$, we use  $\mathcal O(X)$ to denote the topology of $X$.  For any subset $A$ of  $X$, the  \emph{saturation} of $A$, denoted by $Sat(A)$, is defined to be
$$Sat(A) =\bigcap\{U\subseteq \mathcal O(X): A\subseteq U\}.$$
A subset $A$ of a $T_0$ space $X$ is  \emph{saturated} if $A=Sat(A)$.

\begin{remark}(\cite{redbook,goubault}) Let $X$ be a $T_0$ space.
\II
\I[(1)]	  For any subset $A$ of $X$,  $Sat(A)=\ua A$.
\I[(2)]  For any $x\in X$, $\da x=\cl(x)$, and
	$x\in Sat(A)$ if and only if $ \ \da x\cap A\neq\emptyset.$
\I[(3)] For any open subset $U$ of $X$, $U=Sat(U)$, and for any closed subset $F$ of $X$, $F=\da F$.
\III
\end{remark}

	A nonempty subset $A$ of a $T_0$ space is called \emph{irreducible} if for any closed sets $F_1$, $F_2$, $A \subseteq F_1\cup F_2$ implies $A \subseteq F_1$ or $A \subseteq  F_2$.
	
	A $T_0$ space $X$ is called \emph{sober}, if for any  irreducible closed set $F$ of $X$ there is a (unique) point $x\in X$ such that $F=\da x$.

A very  effective tool for  studying  sober spaces is the $b$-topology introduced by L. Skula \cite{skula} (see also \cite{dowker}).

\begin{definition}(\cite{dowker,skula})
Let $X$ be a $T_0$ space. The \emph{$b$-topology} associated with $X$ is the topology which has the family  
$$\{U\cap\cl(x): x\in U\in\mathcal O(X)\}$$  as a base. The  resulting space is denoted by $bX$.   A subset $B$ of $X$  is \emph{$b$-dense} in $X$, if it is dense in $X$ with respect to the $b$-topology.
\end{definition}

In what follows, we shall denote by ${\bf Top_0}$ (resp. ${\bf Top_1}$, {\bf Sob}) the category of all $T_0$ spaces (resp. $T_1$ spaces, sober spaces) with  continuous mappings as morphisms.  All subcategories   of ${\bf Top_0}$  are assumed to be full and  closed under the formation of homeomorphic objects (i.e., satisfy (K2)). Note that condition (K2) is generally easy to be verified.

\begin{definition}
A  subcategory ${\bf K}$  of ${\bf Top_0}$ is \emph{reflective}, if for each $X\in {\bf Top_0}$, there exists $X^k\in{\bf K}$ (the \emph{${\bf K}$-completion} for $X$) and a  continuous mapping $\mu_X:X\longrightarrow X^k$ (the \emph{${\bf K}$-reflection} for $X$) satisfying the universal property: for any
continuous mapping $f:X\longrightarrow Z$ to a space $Z\in{\bf K}$, there exists a unique continuous mapping $g:X^k\longrightarrow Z$ such that $g\circ \mu_X=f$:
\begin{equation*}
\xymatrix@C=5em@R=5ex{
	X \ar[dr]_-{f} \ar[r]^-{\mu_X}
	&X^k\ar@{-->}[d]^-{g}\\
	&Z}
\end{equation*}			
\end{definition}	
	
Equivalently,   ${\bf K}$ is reflective if the inclusion functor $F: {\bf K}\rightarrow {\bf Top_0}$ has a right adjoint (see IV-3 in \cite{maclane}). 

\begin{definition}	
A  mapping  $e:X\longrightarrow Y$  between topological spaces is called  a \emph{$b$-dense embedding}, if it is a topological  embedding such that $e(X)$ is $b$-dense in $Y$.
\end{definition}

\begin{remark}
	The category ${\bf Sob}$ is a full reflective subcategory of ${\bf Top_0}$, and  each reflection $\mu_X$ is a $b$-dense embedding (see \cite[3.1.2]{hof}, or \cite[Proposition 3.2]{KeimelLawson}). The ${\bf Sob}$-completion for $X$ is usually called the \emph{sobrification} of $X$.
\end{remark}

\begin{remark}\label{rem2}
By the definition of  ${\bf K}$-completion, we easily deduce the following:
\II	
\I[(1)]	if both $\mu_1:X\longrightarrow Y_1$ and $\mu_2:X\longrightarrow Y_2$ are ${\bf K}$-reflections for $X$, then there exists a (unique) homeomorphism $h:Y_1\longrightarrow Y_2$ such that $h\circ \mu_1=\mu_2$;
\I[(2)]	if $\eta: X\longrightarrow X^k$ is a ${\bf K}$-reflection for $X$, and $h:X^k\longrightarrow Y$ is a homeomorphism to some $T_0$ space $Y$, then
		$h\circ\eta:X\longrightarrow Y$ is a ${\bf K}$-reflection for $Y$.
\III
\end{remark}





\begin{theorem}[\cite{KeimelLawson}]\label{th11}
Let $X$ be a sober space and $A\subseteq X$. 
\II	
\I[$(1)$] The subspace $A$  is  sober  if and only if $A$ is $b$-closed;
\I[$(2)$] the inclusion mapping $e:A\longrightarrow A^s$, $x\mapsto x$, is a sober reflection for $A$, where $A^s$ is the $b$-closure of $A$ in $X$.
\III
\end{theorem}

\begin{theorem}[\cite{KeimelLawson}]\label{th12}
Let $X$ be a $T_0$ space,  $Y$  a sober space and $f:X\longrightarrow Y$ a continuous mapping. Then $f$ is a sober reflection for $X$ if and only if $f$ is a $b$-dense embedding.
\end{theorem}


A space $X$ is  a \emph{retract} of space $Y$,  if there are two continuous maps $s:X\longrightarrow Y$ (the \emph{section}) and $r:Y\longrightarrow X$ (the \emph{retraction}) such that $r\circ s={\rm id}_X$, the identity mapping on $X$.
	We call $X$ a \emph{$b$-retract} of $Y$  if $X$ is a retraction of $Y$ such that $s(X)$ is $b$-dense in $Y$.

\begin{remark}\label{rem1}
Every section $s:X\longrightarrow Y$ is an embedding and every retraction $r:Y\longrightarrow X$ is a quotient mapping.
\end{remark}

\begin{proposition}[\cite{skula,shen1}]\label{prop1}
 If  $X$ and $Y$ are $T_0$ spaces and  $X$ is a $b$-retract of $Y$, then $X$ is homeomorphic to $Y$.
\end{proposition}

\begin{theorem}[\cite{shen1}]\label{th101}
Let ${\bf K}$ be a reflective subcategory of ${\bf Top_0}$ such that ${\bf K}\nsubseteq {\bf Top_1}$.  Then each ${\bf K}$-reflection is a $b$-dense embedding.
\end{theorem}

\section{Main results}
We now prove our main results.

Recall that the \emph{Sierpi\'{n}ski space}  is the Scott space  $\Sigma 2$ of the two elements chain $2=\{0,1\}$.

\begin{remark} [\cite{redbook,goubault}]\label{rem90}
\II	
\I[(1)] $(\Sigma 2)^M=\Sigma (2^M,\subseteq)$.
\I[(2)] Let $X$ be a $T_0$ space and $M=\mathcal O(X)$.
Then the mapping $e:X\longrightarrow(\Sigma 2)^M$,  $x\mapsto (\chi_{U}(x))_{U\in\mathcal O(X)}$, is an embedding. Hence, by Theorem \ref{th11}, $X$ is a sober space iff $e(X)$ is a $b$-closed subspace of $(\Sigma 2)^M$.
\III
\end{remark}

\begin{lemma}\label{lem93}
	Let $X$ be a $T_0$ space, $x_1, x_2\in X$. If $x_1 < x_2$, then $A=\{x_1,x_2\}$ is a $b$-closed subspace of $X$ that is homeomorphic to $\Sigma 2$.
\end{lemma}
\begin{proof}
	It suffices to prove that $X\setminus\{x_1,x_2\}$ is $b$-open. Let $x\in X\setminus\{x_1,x_2\}$.
	We need to find an open neighborhood $U$ of $x$ such that $\da x\cap U\cap \{x_1,x_2\}=\emptyset$.
	If $\da x\cap \{x_1,x_2\}=\emptyset$, then just need to  take $U=X$.
	Now assume $\da x\cap \{x_1,x_2\}\neq\emptyset$. We consider the following cases:
	
	(c1) $x_1,x_2\in \da x$, i.e., $x_1<x_2< x$. We let $U=X\setminus\da x_2$, then $x\in U$, and $\{x_1,x_2\}\cap U=\emptyset$, and hence $\da x\cap U\cap \{x_1,x_2\}=\emptyset$.
	
	(c2) $x_1\in\da x$ and $x_2\notin \da x$, i.e., $x_1< x$ but $x_2\nleq x$. Let $U=X\setminus\da x_1$. Then $x\in U$ and $\da x\cap U\cap \{x_1,x_2\}=\emptyset$.
	
	All these together show that $X\setminus\{x_1,x_2\}$ is $b$-open, and thus $\{x_1,x_2\}$ is $b$-closed. Note that $\{x_2\}=(X\setminus\da x_1)\cap\{x_1,x_2\}$, which is the unique non-trivial open set in $\{x_1,x_2\}$. Hence, the subspace $\{x_1,x_2\}$ is homeomorphic to $\Sigma 2$.
\end{proof}

The following lemma generalizes a result in  \cite[I-2.5]{skula}.

\begin{lemma}  \label{lem1} 
	Let $X,Y,Z\in {\bf Top_0}$, $k:X\longrightarrow Y$ be a continuous mapping such that $f(X)$ is $b$-dense in $Y$, and $f:X\longrightarrow Z$ be a continuous mapping.	
\II	
\I[$(1)$] There exists at most one continuous mapping $g:Y\longrightarrow Z$ such that $f=g\circ k$.
\I[$(2)$] If $g:Y\longrightarrow Z$ is a continuous mapping such that $f=g\circ k$, then $g(Y)\subseteq \cl_{b}(f(X))$, where $\cl_{b}(f(X))$ is the $b$-closure of $f(X)$ in $Z$.
\III	
\end{lemma}
\begin{proof}
(1)	Suppose that there exist two continuous mappings $g_1,g_2:Y\longrightarrow Z$ such that $g_1\circ k=g_2\circ k=f$:
	$$\xymatrix@C=6em@R=6ex{ 
		X\ar@{->}[rd]_{f}\ar@{->}[r]^{k}&Y\ar@{->}[d]^{g_1, g_2}\\
		&Z.
	}$$               
	Let $y\in Y$. Suppose $V\in\mathcal O(Z)$ such that $g_1(y)\in V$. Then $y\in g_1^{-1}(V)\in\mathcal O(Y)$. It follows that  $\da_Yy\cap g_1^{-1}(V)$ is  $b$-open in $Y$, and since $k(X)$ is $b$-dense in $Y$,  $\da_Y y\cap g_1^{-1}(V)\cap k(X)\neq\emptyset$. Since $g_1\circ k=g_2\circ k=f$, we deduce that $g_1^{-1}(V)\cap k(X)=g_2^{-1}(V)\cap k(X)$. Then $\da_Yy\cap g_2^{-1}(V)\cap k(X)\neq\emptyset$, and hence $y\in \ua_Yg_2^{-1}(V)=g_2^{-1}(V)$, i.e., $g_2(y)\in V$. 
	All these show that each open neighborhood of $g_1(y)$ contains $g_2(y)$, that is, $g_1(y)\in\cl_Z(g_2(y))$. Dually, it holds that $g_2(y)\in\cl_Z(g_1(y))$. Since $Z$ is a $T_0$ space, we have that $g_1(y)=g_2(y)$. Therefore, $g_1=g_2$.
	
	(2)  Let $y\in Y$ and $V\in\mathcal O(Z)$ such that $g(y)\in V$. Then $y\in g^{-1}(V)\in\mathcal O(Y)$, and since $k(X)$ is $b$-dense in $Y$, $\da y\cap g^{-1}(V)\cap k(X)\neq\emptyset$. Then there exists $x_0\in X$ such that $k(x_0)\in \da y\cap g^{-1}(V)$, which implies that $g(y)\geq g(k(x_0))=f(x_0)\in V$. It follows that  $f(x_0)\in \da g(y)\cap V\cap f(X)\neq\emptyset$. This shows that $g(y)\in\cl_{b}(f(X))$. Hence, $g(Y)\subseteq \cl_{b}(f(X))$.
\end{proof}

\begin{theorem}\label{th91}
	Let ${\bf K}$ be a reflective subcategory of ${\bf Top_0}$ such that ${\bf K}\nsubseteq {\bf Top_1}$. Then the following statements hold.
\II	
\I[$(1)$] ${\bf K}$ is $b$-closed-hereditary.
\I[$(2)$] The Sierpi\'{n}ski space $\Sigma 2\in {\bf K}$.
Hence, for any set $M$, the product $(\Sigma 2)^M\in{\bf K}$.
\I[$(3)$] ${\bf Sob}\subseteq {\bf K}$.
\III
\end{theorem}
\begin{proof}

(1)  Let $X\in {\bf K}$, $A$ be a $b$-closed subspace of $X$, and $\mu_A:A\longrightarrow A^k$ be the  ${\bf K}$-reflection for $A$. Then $\mu_A(A)$ is a $b$-dense subset of $A^k$ by Theorem \ref{th101}. Consider the inclusion mapping $e:A\longrightarrow X$, $x\mapsto x$. Then there exists a unique continuous mapping $f: A^k\longrightarrow X$ such that $f\circ \mu_A=e$:
	$$\xymatrix@C=6em@R=6ex{
		A\ar@{->}[r]^{\mu_{A}}\ar@{->}[rd]_{e} &A^k\ar@{-->}[d]^{f}\\
		&	X.
	}$$
	Then by Lemma \ref{lem1}, we have  $f(A^k)\subseteq \cl_{b}(e(A))=A$, which shows that $A$ is a b-dense retract of $A^k$. By Proposition \ref{prop1} $A$ is homeomorphic to $A^k$, and since $A^k\in{\bf K}$, it follows that $A\in {\bf K}$.

(2)	 Since ${\bf K}\nsubseteq {\bf Top_1}$, there exists a $T_0$, but not $T_1$ space $X\in {\bf K}$. Then there exists $x_1,x_2\in X$ such that $x_1<x_2$ (otherwise, each singleton of $X$ is closed, i.e., $X$ is $T_1$). By Lemma \ref{lem93}, the subspace $\{x_1,x_2\}$ is a $b$-closed subspace of $X$, and by result (1), we have $\{x_1,x_2\}\in {\bf K}$, and since the subspace $\{x_1,x_2\}$ is homeomorphic to $\Sigma 2$, which implies $\Sigma 2\in {\bf K}$.
Since ${\bf K}$ is reflective, the product of its objects is also in ${\bf K}$, i.e., ${\bf K}$ is productive, hence we have $(\Sigma 2)^M\in {\bf K}$.	

(3) Let $X\in{\bf Sob}$. By Remark \ref{rem90}, there is an embedding $e:X\longrightarrow (\Sigma 2)^M$ such that $e(X)$ is a $b$-closed subspace of  $(\Sigma 2)^M$. By Step 1, $(\Sigma 2)^M\in{\bf K}$ and since ${\bf K}$ is $b$-closed-hereditary, we have that $e(X)\in {\bf K}$, and since $X$ is homoemorphic to $e(X)$, it follows that $X\in {\bf K}$. Hence, ${\bf Sob}\subseteq {\bf K}$.
\end{proof}

Since  lower subsets and saturated subsets of a topological space are $b$-closed \cite{KeimelLawson,skula}, the following corollary is immediate from Theorem \ref{th91}.
\begin{corollary}
Let ${\bf K}$ be a reflective subcategory of ${\bf Top_0}$ such that ${\bf K}\nsubseteq {\bf Top_1}$. Suppose $X\in {\bf K}$ and $A\subseteq X$.  If $A=\da A$ or $A=\ua A$, then $A$ as a subspace of $X$ is in {\bf K}.
\end{corollary}

Let ${\bf Sier}$ be the full subcategory of ${\bf Top_0}$ consisting of all $T_0$ spaces $X$ which are  homeomorphic to $\Sigma 2$.

Recall that the reflective hull of a subcategory ${\bf C}$ of ${\bf Top_0}$ is the smallest reflective subcategory of ${\bf Top_0}$  containing ${\bf C}$.

\begin{corollary}[\cite{nel}]
	The reflective hull of {\bf Sier} in ${\bf Top_0}$ is ${\bf Sob}$.
\end{corollary}
\begin{proof}
Suppose ${\bf K}$ is a reflective subcategory of ${\bf Top_0}$ such that ${\bf Sier}\subseteq {\bf K}$. Note that $\Sigma 2$ is a $T_0$ but not a $T_1$ space, we have ${\bf K}\nsubseteq {\bf Top_1}$. By Theorem \ref{th91}, ${\bf Sob}\subseteq {\bf K}$. Since ${\bf Sob}$ is  reflective, it is the smallest reflective subcategory of ${\bf Top_0}$ having   {\bf Sier} as a subcategory. Therefore,  {\bf Sob} is the reflective hull of {\bf Sier} in ${\bf Top_0}$.
\end{proof}

\begin{lemma}[{\cite[Lemma 5, pp.116]{kelly}}]
If $\{f_i:X\longrightarrow Y_i\}_{i\in I}$ is a family of continuous mappings in ${\bf Top_0}$, then the diagonal $\varDelta f_i:X\longrightarrow \prod_{i\in I}Y_i$ is a continuous mapping, where
$$\forall x\in X,\ \varDelta f_i(x)=(f_i(x))_{i\in I}.$$ 
\end{lemma}

\begin{theorem}\label{th100}
Let ${\bf K}$	be a reflective subcategory of ${\bf Top_0}$ such that ${\bf K}\nsubseteq {\bf Top_1}$. Then ${\bf K}$ is reflective if and only if it is productive and $b$-closed-hereditary.
\end{theorem}
\begin{proof}
It is well-known that if ${\bf K}$ is reflective, then it is productive (see V-6 in \cite{maclane}), and by Theorem \ref{th91}, it is $b$-closed-hereditary.

Now suppose ${\bf K}$ is productive and $b$-closed-hereditary. Let $X\in {\bf Top_0}$, and let 
$\Phi(X)$ be the family of all pairs $(Y,f)$, where 
$Y\in{\bf K}$ and $f:X\longrightarrow Y$ is a continuous mapping such that $f(X)$ is $b$-dense in $Y$. 
We may denote $\Phi(X)=\{(Y_i,f_i):i\in I\}$.
Let 
$X^k=\cl_b((\varDelta f_i)(X))$ be the $b$-closure of $(\varDelta f_i)(X)$ in the product $\prod_{i\in I}Y_i$, where $\varDelta f_i:X\longrightarrow\prod_{i\in I}Y_i$ is the diagonal.
Since ${\bf K}$ is productive, $\prod_{i\in I}Y_i\in{\bf K}$, and since {\bf K} is $b$-closed-hereditary, $X^k\in {\bf K}$.

 Next, we will prove that $X^k$ as the subspace of $\prod_{i\in I}Y_i$ with  the restriction $k:X\longrightarrow X^k$ of the diagonal $\varDelta f_i$ is the $\mathbf K$-reflection for $X$. To see this, suppose $Y\in {\bf K}$ and $f:X\longrightarrow Y$ is a continuous mapping. We consider the following cases:
 
 {\bf Case 1: }  $f(X)$ is $b$-dense in $Y$, that is, $(Y,f)\in\Phi(X)$, and we may assume $(Y,f)=(Y_j,f_j)$ for some $j\in I$. Let  $p_j:X^k\longrightarrow Y_j$ be the restriction of  the projection from $\prod_{i\in I}Y_i$ to $Y_j$. Then $p_j$ is continuous and it is clear that  $p_j\circ\varDelta f_i= f_j$:
 $$	\xymatrix@C=6em@R=6ex{
 	X\ar@{->}[r] ^{k}\ar@{->}[rd]_{f_j}&X^k\ar@{-->}[d]^{p_j}\\
 	&Y_j.\\}$$
The uniqueness of $p_j$ is immediate from \ref{lem1}.

 {\bf Case 2: } $f(X)$ is not $b$-dense in $Y$. Let $\cl_{b}(f(X))$ be the $b$-closure of  $f(X)$ in $Y$.
 Then the corestriction $\widehat{f}:X\longrightarrow \cl_{b}(f(X))$ ($\widehat{f}=f(x)$) of $f$ is a continuous mapping such that $\widehat{f}(X)$ is $b$-dense in $\cl_{b}(f(X))$. Hence, $(\cl_{b}(f(X)),\widehat{f})\in \Phi(X)$, and we may assume $(\cl_{b}(f(X)),\widehat{f})=(Y_j,f_j)$ for some $j\in I$. Then by case 1, $p_j$ is continuous and it is clear that  $p_j\circ\varDelta f_i= f_j$:
 $$	\xymatrix@C=6em@R=6ex{
	X\ar@{->}[r] ^{k}\ar@{->}[rd]_{f_j}&X^k\ar@{-->}[d]^{p_j}\\
	&Y_j.\\}$$
Note that $e\circ f_j=f$, where $e:\cl_{b}(f(X))\longrightarrow Y$ is the inclusion mapping. Then the composition $e\circ p_j:X^k\longrightarrow Y$ is continuous and $(e\circ p_j)\circ k=e\circ(p_j\circ k)=e\circ f_j=f$:
 $$	\xymatrix@C=4em@R=6ex{
	X\ar@{->}[ddr] _{f}\ar@{->}[rr] ^{k}\ar@{.>}[rd]^{f_j}&&X^k\ar@{.>}[dl]_{p_j}\ar@{->}[ddl]^{e\circ p_j}\\
	&Y_j\ar@{.>}[d] ^{e}&\\
&Y.&
}$$
The uniqueness of $e\circ p_j$ is immediate from \ref{lem1}.

All this shows that $k:X\longrightarrow X^k$ is the {\bf K}-reflection for $X$.
\end{proof}

\begin{definition}
	We say that a category  {\bf K}  \emph{has equalizers} if for any morphisms $f,g:X\longrightarrow Y$ in {\bf K}, the equalizer $E_{f,g}$ of $f$ and $g$ belongs to ${\bf K}$.
\end{definition}

\begin{remark}
	Let ${\bf K}$ be a subcategory of ${\bf Top_0}$ and $X,y\in{\bf K}$. If  $f,g:X\longrightarrow Y$ are continuous mappings, then the equalizer of $f$ and $g$ in ${\bf K}$ can be identified with the subspace $E_{f,g}=\{x\in X: f(x)=g(x)\}$ of $X$. 
\end{remark}

\begin{lemma}\label{th007}
	Let $X\in{\bf Top_0}$ and $E\subseteq X$. Then the following are equivalent:
\II
\I[$(1)$]  $E$ is $b$-closed in $X$;
\I[$(2)$]  there exist continuous mappings $f,g:X\longrightarrow (\Sigma 2)^M$ for some set $M$ such that  $E=\{x\in X: f(x)=g(x)\}$;
\I[$(3)$]  there exist continuous mappings $f,g:X\longrightarrow Y$ for some $Y\in{\bf Top_0}$ such that  $E=\{x\in X: f(x)=g(x)\}$.
\III	
\end{lemma}
\begin{proof}
	(1) $\Rightarrow$ (2). Note that the family $\{U\cup (X\setminus V): U,V\in\mathcal O(X)\}$ forms a closed base for the $b$-topology. Since $E$ is $b$-closed, 
	$$E=\bigcap_{i\in M} U_i\cup (X\setminus V_i),$$
	where $U_i,V_i\in\mathcal O(X)$ for all $i\in M$. 
	Define $f,g:X\longrightarrow (\Sigma 2)^M$ by 
	$$f(x)(i)=\chi_{U_i}(x)\ \text{ and } g(x)(i)=\chi_{U_i\cup V_i}(x)$$
	for all $i\in M$. It is easy to check that  both $f$ and $g$ are continuous, and that $f(x)(i)=g(x)(i)$ iff $x\in U_i\cup (X\setminus V_i)$. It follows that  $E=\{x\in X:f(x)=g(x)\}$. 
	
	\medskip
	
	(2) $\Rightarrow$ (3). This is trivial.
	
	\medskip
	
	(3) $\Rightarrow$ (1). Let $x\notin E$. Then $f(x)\neq g(x)$. Since $Y$ is $T_0$, we may assume $f(x)\nleq g(x)$ without loss of generality. Then there exists $V\in\mathcal O(Y)$ such that $f(x)\in V$ and $g(x)\notin V$. It follows that $x\in f^{-1}(V)$ and $x\notin g^{-1}(V)$.  We claim that $\da x\cap f^{-1}(V)\cap E=\emptyset$. Otherwise, if $y\in \da x\cap f^{-1}(V)\cap E$, then $g(y)=f(y)\in V$ and $g(y)\leq g(x)$, and hence $g(x)\in V$, a contradiction.
\end{proof}

\begin{proposition}\label{th212}
	Let ${\bf K}$ be a subcategory of ${\bf Top_0}$. 
	If  $\{(\Sigma 2)^M:M\text{ is a set}\}\subseteq {\bf K}$, then the following are equivalent:
\II
\I[$(1)$]  ${\bf K}$ has equalizers;
\I[$(2)$]  ${\bf K}$ is $b$-closed-hereditary.
\III
\end{proposition}
As an immediate result of Theorem \ref{th100} and Proposition \ref{th212}, we have the following  Theorem.
\begin{theorem}[\cite{nel}]\label{th213}
	Let ${\bf K}$	be a reflective subcategory of ${\bf Top_0}$ such that ${\bf K}\nsubseteq {\bf Top_1}$. Then ${\bf K}$ is reflective if and only if it is productive and has equalizers.
\end{theorem}

\begin{theorem}\label{k3}
Let ${\bf K}$	be a reflective subcategory of ${\bf Top_0}$ such that ${\bf K}\nsubseteq {\bf Top_1}$. If $\{X_i:i\in I\}\subseteq{\bf K}$ is a family of subspaces of a sober space $Z$, then
\II
\I[$(1)$] the subspace $\bigcap_{i\in I}X_i\in{\bf K}$;
\I[$(2)$]  for each subspace $X$ of  $Z$,  the inclusion mapping $e_X:X\longrightarrow\cl_{k}(X)$ is a ${\bf K}$-reflection for $X$, where $\cl_{k}(X)=\bigcap\{A\in{\bf K}:X\subseteq A\subseteq Z\}$.
\III
\end{theorem}
\begin{proof}
(1) Let $X=\bigcap_{i\in I}X_i$. Then the inclusion mapping $e:X\longrightarrow X^s=\cl_{b}(X)$, $x\mapsto x$, is a sober reflection for $X$. Note that $X=\bigcap_{i\in I}X_i\cap X^s$, we may assume all $X_i\subseteq X^s$ without loss of generality.
Assume $\mu_{X}:X\longrightarrow X^k$ is a {\bf K}-reflection for $X$.
By Theorem \ref{th91}, $X^s\in {\bf K}$, and thus there exists a unique continuous mapping $f:X^k\longrightarrow X^s$ such that $f\circ\mu_{X}=e$:
$$	\xymatrix@C=6em@R=6ex{
	X\ar@{->}[r] ^{\mu_X}\ar@{->}[rd]_{e}&X^k\ar@{-->}[d]^{f}\\
	&X^s.\\}$$

Claim :  $f(X^k)=X$.

It suffices to prove $f(X^k)\subseteq X_i$ for all $i\in I$.
Note that $X_i\in{\bf K}$, and let $e^i:X\longrightarrow X_i$ and $e_i:X_i\longrightarrow X^s$ be the inclusion mappings. It is clear that $e=e_i\circ e^i$: 
$$	\xymatrix@C=6em@R=6ex{
	X\ar@{->}[r] ^{e^i}\ar@{->}[rd]_{e}&X_i\ar@{->}[d]^{e_i}\\
	&X^s.\\}$$
Then there exists $f_i:X^k\longrightarrow X_i$ such that $f_i\circ \mu_{X}=e^i$:
$$	\xymatrix@C=6em@R=6ex{
	X\ar@{->}[r] ^{\mu_X}\ar@{->}[rd]_{e^i}&X^k\ar@{-->}[d]^{f_i}\\
	&X_i.\\}$$
Now $e_i\circ f_i:X^k\longrightarrow X^s$ is a continuous mapping such that  $(e_i\circ f_i)\circ\mu_{X}=e_i\circ(f_i\circ \mu_{X})=e_i\circ e^i=e=f\circ\mu_{X}$, and then by the uniqueness of $f$, we have that $e_i\circ f_i=f$:
$$	\xymatrix@C=6em@R=6ex{
	X\ar@{->}[r] ^{\mu_X}\ar@{->}[rd]_{e^i}&X^k\ar@{->}[d]^{f_i}\ar@{->}[rd]^{f=e_i\circ f_i}&\\
	&X_i\ar@{->}[r]_{e_i}&X^s.\\
}$$
For each $y\in X^k$, we have $f(y)=e_i(f_i(y))=f_i(y)\in X_i$, hence $f(X^k)\subseteq X_i$. Therefore, $f(X^k)\subseteq\bigcap_{i\in I}X_i=X$, and the reverse inclusion is trivial, showing the claim.

\medskip

Consider the co-restriction $\widehat{f}:X^k\longrightarrow f(X^k)=X$ of $f$, which is a continuous mapping such that $\widehat{f}\circ\mu_{X}={\rm id}_X$, the identity mapping on $X$. This implies $X$ is a $b$-retract of $X^k$, and by Proposition \ref{prop1},  $X$ is homeomorphic to $X^k\in{\bf K}$. Therefore, $X\in{\bf K}$.

\medskip

(2)  Let $X^s$ is the $b$-closure of $X$ in $Z$. Then the inclusion mapping $e:X\longrightarrow X^s$ is a sober reflection for $X$.
Next, we prove the conclusion in some steps.

 Step 1. \
Suppose that $\mu_{X}:X\longrightarrow X^k$ is the {\bf K}-reflection for $X$. Since $X^s\in{\bf Sob}\subseteq {\bf K}$ by Theorem \ref{th91},  there exists a unique continuous mapping $f:X^k\longrightarrow X^s$ such that $f\circ\mu_{X}=e:$
$$	\xymatrix@C=6em@R=6ex{
	X\ar@{->}[r] ^{\mu_X}\ar@{->}[rd]_{e}&X^k\ar@{-->}[d]^{f}\\
	&X^s.\\}$$
Using a similar proof of the Claim in (1), one can prove that $$f(X^k)\subseteq \bigcap\{A\in{\bf K}:X\subseteq A\subseteq X^s\}=\bigcap\{A\in{\bf K}:X\subseteq A\subseteq Z\}=\cl_{k}(X).$$

Step 2.\
Suppose that $\eta_{X^k}:X^k\longrightarrow Y$ is a sober reflection for $X^k$. By Theorems \ref{th12} and  \ref{th101}, both $\mu_{X}$ and $\eta_{X^k}$ are $b$-dense embedding, so is their composition $\eta_{X^k}\circ\mu_{X}:X\longrightarrow Y$, which is a sober reflection for $X$ by Theorem \ref{th12}.
Then there exists a unique homoemorphism $h:Y\longrightarrow X^s$ such that $h\circ\eta_{X^k}\circ\mu_{X}=e$. We have $(h\circ \eta_{X^k})\circ\mu_X=e=f\circ\mu_{X}$, which implies $f=h\circ\eta_{X^k}$ by the uniqueness of $f$, i.e., the following diagram commutes:
$$	\xymatrix@C=6em@R=6ex{
	X\ar@{->}[r] ^{\mu_X}\ar@{->}[rd]_{e}&X^k\ar@{->}[d]^{f}\ar@{->}[r]^{\eta_{X^k}}&Y\ar@{-->}[dl]^{h}\\
	&X^s.\\}$$
Since $h$ is a homeomorphism and $\eta_{X^k}$ is a $b$-dense embedding, we deduce that $f=h\circ\eta_{X^k}$ is also a $b$-dense embedding (which is actually a sober reflection for $X^k$). It follows that $f(X^k)$, as a subspace of $X^s$,  is homeomorphic to $X^k\in {\bf K}$, and hence $f(X^k)\in {\bf K}$. Since $X\subseteq f(X^k)\subseteq X^s$, we have that $\bigcap\{A\in{\bf K}:X\subseteq A\subseteq Z\}\subseteq f(X^k)$.
By Step 1, we have that $$f(X^k)=\bigcap\{A\in{\bf K}:X\subseteq A\subseteq Z\}=\cl_{k}(X),$$
which implies the co-restriction $\widehat{f}:X^k\longrightarrow \cl_k(X)$ of $f$ is a homeomorphism such that $\widehat{f}\circ \mu_{X}=e$:
$$	\xymatrix@C=6em@R=6ex{
	X\ar@{->}[r] ^{\mu_X}\ar@{->}[rd]_{e}&X^k\ar@{->}[d]^{\widehat{f}\text{ (a homeomorphism)}}\\
	&\cl_{k}(X).\\}$$
 Therefore, $e:X\longrightarrow \cl_{k}(X)$ is  also a {\bf K}-reflection for $X$.
\end{proof}

\begin{theorem}\label{k4}
Let ${\bf K}$	be a reflective subcategory of ${\bf Top_0}$ such that ${\bf K}\nsubseteq {\bf Top_1}$. If $f:X\longrightarrow Y$ is a continuous mapping from a sober space $X$ to a sober space $Y$, then for any subspace $Y_1$ of $Y$, $Y_1\in{\bf K}$ implies that $f^{-1}(Y_1)\in {\bf K}$.
\end{theorem}
\begin{proof}
Let $X_1=f^{-1}(Y)$ and $(X_1)^k=\bigcap\{K\in {\bf K}: X_1\subseteq K\subseteq X\}$ be the subspace of $X$. By Theorem \ref{k3}, the inclusion mapping $e_1:X_1\longrightarrow (X_1)^k$ is a {\bf K}-reflection for $X_1$. Consider the restriction $f_1:X_1\longrightarrow Y_1$ ($x\mapsto f(x)$) of $f$, then there exists a unique continuous mapping $g_1:(X_1)^k\longrightarrow Y_1$ such that $g_1\circ e_1=f_1$:
$$	\xymatrix@C=6em@R=6ex{
X_1\ar@{->}[r] ^{e_1}\ar@{->}[rd]_{f_1}&(X_1)^k\ar@{-->}[d]^{g_{1}}\\
	&Y_1.\\}$$

Consider the composition $e_Y\circ f_1:X_1\longrightarrow Y$, where  $e_{Y}:Y_1\longrightarrow Y$ is the inclusion mapping. There exists a unique continuous mapping $g_2:(X_1)^k\longrightarrow Y$ such that $g_2\circ e_1=e_Y\circ f_1$:
$$	\xymatrix@C=6em@R=6ex{
	X_1\ar@{->}[r] ^{e_1}\ar@{->}[rd]_{f_1}&(X_1)^k\ar@{-->}[rd]^{g_{2}}&\\
	&Y_1\ar@{->}[r]^{e_{Y}}&Y.\\}$$

Let $f_2:(X_1)^k\longrightarrow Y$ ($x\mapsto f(x)$) be the restriction of $f$. On the one hand, for each $x\in X_1$, we have $(f_2\circ e_1)(x)=f(x)=(e_Y\circ f_1)(x)=(g_2\circ e_1)(x)$, it follows that $f_2\circ e_1=g_2\circ e_1$, which implies $g_2=f_2$ by the uniqueness of $g_2$. On the other hand, $g_2\circ e_1=e_Y\circ f_1=e_Y\circ (g_1\circ e_1)=(e_Y\circ g_1)\circ e_1$, which implies that $e_Y\circ g_1=g_2=f_2$ by the uniqueness of $g_2$, i.e., the following diagram commutes:
$$	\xymatrix@C=6em@R=6ex{
	X_1\ar@{->}[r] ^{e_1}\ar@{->}[rd]_{f_1}&(X_1)^k\ar@{->}[d]^{g_{1}}\ar@{->}[rd]^{g_{2}=f_2}&\\
	&Y_1\ar@{->}[r]^{e_{Y}}&Y.\\}$$

 Then for each $x\in (X_1)^k$, we have $f(x)=f_2(x)=g_2(x)=(e_Y\circ g_1)(x)=g_1(x)\in Y_1$, which implies $x\in f^{-1}(Y_1)=X_1$. Hence, $(X_1)^k\subseteq X_1$, and so $X_1=(X_1)^k\in{\bf K}$.
\end{proof}

Using Theorems \ref{th91}, \ref{k3} and \ref{k4}, we obtain the main result in this paper.
\begin{theorem}\label{th011}
	Let ${\bf K}$ be a subcategory of ${\bf Top_0}$ such that ${\bf K}\nsubseteq{\bf Top_1}$. Then 	${\bf K}$ is reflective in ${\bf Top_0}$	
if and only if 	 ${\bf K}$ satisfies the conditions (K1)--(K4).
\end{theorem}

As an immediate result of Theorems \ref{th100}, \ref{th213} and \ref{th011}, the characterizations for the reflectivity of ${\bf K}$  can be summarized as follows.
\begin{corollary}\label{th00}
	Let ${\bf K}$ be a subcategory of ${\bf Top_0}$ such that ${\bf K}\nsubseteq{\bf Top_1}$. Then the following statements are equivalent:
\II	
\I[$(1)$]	${\bf K}$ is reflective in ${\bf Top_0}$;	
\I[$(2)$]	 ${\bf K}$ satisfies conditions (K1)--(K4);	
\I[$(3)$]	 ${\bf K}$ is productive and $b$-closed-hereditary;
\I[$(4)$]	${\bf K}$ is productive and has equalizers.
\III
\end{corollary}

\section{Some applications}

By using the results in the last section, we now investigate  the reflectivity  of the categories of co-sober spaces, strong $d$-spaces,
$k$-bounded sober spaces, and open well-filtered spaces.
Note that all these classes of spaces  are closed under the formation of homeomorphic objects.

\subsection{Co-sober spaces}
In \cite{defcosob}, to study the dual Hofmann-Mislove Theorem, Escard\'{o},
Lawson and Simpson introduced the co-sober spaces  \cite{defcosob}, which are defined below.
\begin{definition}[\cite{defcosob}] Let $X$ be a $T_0$ space. and $Q$ be a compact saturated subset of $X$.
\II	
\I[(1)] $Q$ is called \emph{$k$-irreducible} if for any compact saturated subsets $Q_1,Q_2$ of $X$, $Q=Q_1\cup Q_2$ implies $Q=Q_1$ or $Q=Q_2$.
\I[(2)] $X$ is called \emph{co-sober} if for each $k$-irreducible set $Q$, there exists a unique $x\in X$ such that  $Q=\ua x$.
\III
The category of all co-sober spaces with continuous mappings is denoted by ${\bf Co}$-${\bf Sob}$. Note that   ${\bf Co}$-${\bf Sob}$ is a subcategory of ${\bf Top_0}$.
\end{definition}

 \begin{example}\label{exa}
Let  $\mathbb N$ be the space of all natural numbers with the Alexandorff topology (the open sets are $\emptyset$, $\mathbb N$ and all the sets of form $\ua n$, $n\in\mathbb N$). It is clear that the compact saturated sets are of the form $\ua n$, $n\in N$. Hence, the Alexandorff topology on $\mathbb N$ is co-sober, and  not $T_1$.
\end{example}
From Example \ref{exa}, we  have that ${\bf Co}$-${\bf Sob}\nsubseteq {\bf Top_1}$.
In \cite{defcosob}, it is asked whether every  sober space is co-sober.
 In \cite{wexu}, Wen and Xu gave a negative answer, by proving that the Isbell’s complete lattice  (see \cite{isbell}) equipped with the lower topology is sober but not co-sober. Hence,  ${\bf Sob}\nsubseteq{\bf Co}$-${\bf Sob}$.
 Then by Theorem \ref{th00}, we obtain the following result.
\begin{corollary}
	The category ${\bf Co}$-${\bf Sob}$ is not reflective in ${\bf Top_0}$.
\end{corollary}

\subsection{Strong $d$-spaces}

The  strong $d$-spaces were  introduced by Xu and Zhao \cite{strongd}, which lie  between the classes of  $T_1$ spaces and that of $d$-spaces.

\begin{definition}[\cite{strongd}]
	A $T_0$ space $X$ is called a \emph{strong $d$-space} if for any $x\in X$, directed subset $D$ of $X$ and open subset $U$ of $X$,
	$\bigcap_{d\in D}\ua d\cap \ua x\subseteq U$ implies $\ua d_0\cap \ua x\subseteq U$ for some $d_0\in D$.
	
	The category of all strong $d$-spaces with continuous mappings is denoted by ${\bf SD}$. Then ${\bf SD}$ is a subcategory of ${\bf Top_0}$.
\end{definition}

It has been  shown that 	there exists a continuous dcpo $P$ whose Scott topology is not strong $d$-space \cite[Example 3.34]{strongd}, and the Scott topology on every continuous lattice is a strong $d$-space \cite[Remark 3.21]{strongd}. Hence, ${\bf Sob}\nsubseteq {\bf SD}$ and ${\bf SD}\nsubseteq{\bf Top_1}$. Then by Theorem \ref{th00}, we deduce the following result.

\begin{corollary}
	The category ${\bf SD}$  is not reflective in ${\bf Top_0}$.
\end{corollary}

\subsection{$k$-bounded sober spaces}
In \cite{sitopology}, Zhao and Ho introduced another weaker notion of sobriety, \emph{the k-bounded sobriety}.
\begin{definition}[\cite{sitopology}]
 A $T_0$ space $X$ is  \emph{$k$-bounded sober}  if for any irreducible closed subset $F$ of $X$ whose  $\bigvee F$ exists, there is a unique point $x\in X$ such that $F=\da x$.
\end{definition}
 The category of all $k$-bounded sober spaces with continuous mappings is denoted by ${\bf KSob}$. Then  ${\bf KSob}$ is a subcategory of ${\bf Top_0}$. Clearly, ${\bf Sob}\subseteq{\bf KSob}$, and since there exists a sober space $X$ (hence is co-sober) that is not $T_1$, we have that ${\bf KSob}$$\nsubseteq {\bf Top_1}$.

\begin{example}\label{exm00}
	Let $X=[0,3]$ with the Scott topology (i.e., the open sets are $\emptyset$, $[0,3]$ and all sets of the form $(x,3]$, $x\in[0,3]$). Since $[0,3]$ is a continuous lattice, $X$ is a sober space. For each integer $n\geq 2$, let $X_n=[0,1)\cup(2-\frac{1}{n}, 2+\frac{1}{n})$. We have the following facts.
\II
\I[(1)] Each $X_n$ is a $k$-bounded sober subspace of $X$.
	
	Let $F$ be an irreducible closed set in $X_n$ such that $\bigvee_{X_n}F=x$ exists. There are two cases: (c1) $x\in[0,1)$. Then $F\subseteq [0,1)$, and clearly we have that
	$\cl_{X_n}(F)=\cl_{X_n}(\{x\})$.  (c2) $x\in (2-\frac{1}{n}, 2+\frac{1}{n})$. Then  $F\cap (2-\frac{1}{n}, 2+\frac{1}{n})\neq\emptyset$, and thus $x=\bigvee_{X_n}F=\bigvee_{X_n} F\cap(2-\frac{1}{n}, 2+\frac{1}{n})=\bigvee_X F\cap(2-\frac{1}{n}, 2+\frac{1}{n})=\bigvee_X F$. Since $X$ is sober, we have $\cl_X(F)=\cl_X(\{x\})$, and thus $\cl_{X_n}(F)=\cl_X(F)\cap X_n=\cl_X(\{x\})\cap X_n=\cl_{X_n}(\{x\})$. All this shows that $X_n$ is $k$-bounded sober.
	
\I[(2)] The intersection $Y=\bigcap_{n\geq 2}X_n=[0,1)\cup\{2\}$ is not $k$-bounded sober.
	
	 Let $F=[0,1)$.  Then $\bigvee_{Y}F=2$ and $F$ is irreducible since it is a directed set.  Note that $[0,1]$ is closed in $X$ and $F=[0,1]\cap Y$, so $F$ is a closed set in $Y$. For each $x\in [0,1)$, we have $\cl_{Y}(\{x\})=[0,x]\neq F$, and $\cl_{Y}(\{2\})=Y\neq F$. Therefore, $Y$ is not a $k$-bounded sober space.
\III
\end{example}
	

From Example \ref{exm00}, the category ${\bf KSob}$ does not satisfies (K3), hence by Theorem \ref{k3}, we have the following result.

\begin{corollary}[see also in \cite{kbsober}]
	The category ${\bf KSob}$  is not reflective in ${\bf Top_0}$.
\end{corollary}

\subsection{Open well-filtered spaces}

 The open well-filtered spaces were introduced by Shen, Xi, Xu and Zhao, which form a proper larger class than that of  the well-filterdness spaces \cite{shenowf}.
One  notable result on open  well-filtered spaces is that every core-compact open well-filtered space is sober.

\begin{definition}\cite{shenowf}
Let $X$ be a $T_0$ space.
\II	
\I[(1)]	 $\mathcal F\subseteq \mathcal O(X)$ is a $\ll$-filtered family if for any  $U_1,U_2\in\mathcal F$, there exists $U_3\in\mathcal F$ such that
	$U_3\ll U_1,U_2$ in the poset $(\mathcal O(X),\subseteq)$.
\I[(2)] $X$ is called \emph{open well-filtered} if for each $\ll$-filtered family $\mathcal F\subseteq \mathcal O(X)$  and $U\in\mathcal O(X)$, $$\bigcap\mathcal F\subseteq U\ \Rightarrow\ \ V\subseteq U\text{ for some }V\in\mathcal F.$$
\III
The category of all open well-filtered spaces with continuous mappings is denoted by {\bf OWF}.
Thus {\bf OWF} is a full subcategory of ${\bf Top_0}$.
\end{definition}

\begin{lemma}
	Let $X=\mathbb N$ be the set of positive integers with the cofinite topology (i.e., the open sets are $\emptyset$, $X$ and all sets of form $X\setminus F$, where $F$ is a finite subset of $X$). Then $X$ is not open well-filtered.
\end{lemma}
\begin{proof}
	First, note that each subset of $X$ is compact. Then for any open subsets $U,V$ of $X$, $U\ll V$ iff $U\subseteq V$. Then the family
	$\{X\setminus\da n:n\in X\}$ is a $\ll$-filtered family of open sets, where $\da n=\{k:k\leq n\}$.
	We have $\bigcap_{n\in X}X\setminus\da n=\emptyset$ but $X\setminus\da n\nsubseteq\emptyset$ for any $n\in X$. Therefore, $X$ is not an open well-filtered space.
\end{proof}

\begin{example}
	Consider Johnstone's dcpo $L=\mathbb N\times(\mathbb N\cup\{\infty\})$, i.e., with the partial order by
	$(m,n)\leq (m',n')$ iff either $m=m'$ and $n\leq n'\leq\infty$ or $n'=\infty$ and $n\leq m'$,  shown in Figure 1:
	\begin{figure}[htbp]
		\centering
		\includegraphics[scale=.45]{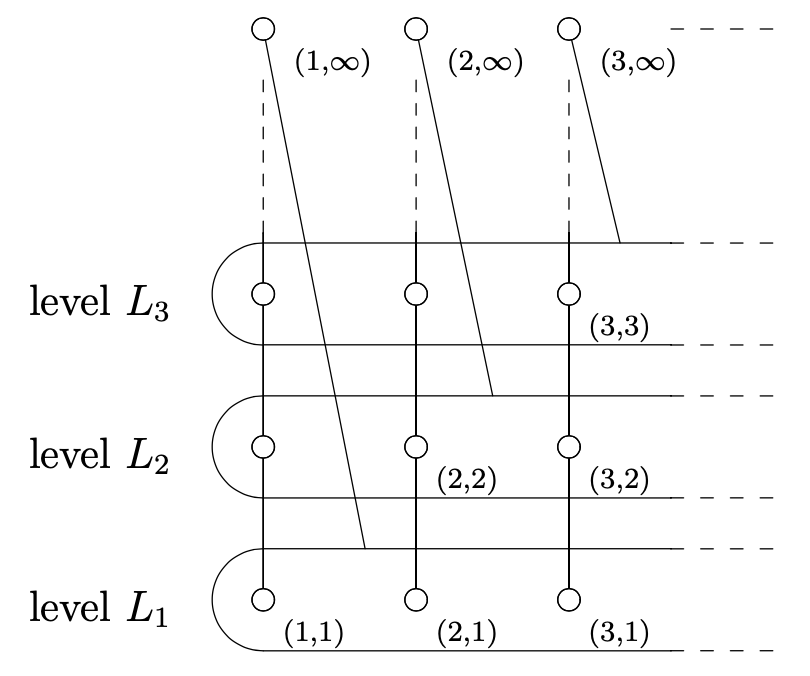}
		\caption{The Johnstone's dcpo}
	\end{figure}
	
	(1) $\forall U,V\in \sigma(L)$, $U\ll V$ iff $U=\emptyset$.
	
	Suppose $U\ll V$ and $U\neq\emptyset$.
	Since $U$ is Scott open, all but only finite  maximal elements are in $U$, i.e., $\{(n,\infty):n\in\mathbb N,(n,\infty)\notin U\}$ is finite. There exists $M\in\mathbb N$ such that
	$(n,\infty)\in U$ for all $n>M$. For each $n>M$, since $(n,\infty)=\bigvee_{m\in\mathbb N}(n,m)$, it follows that $\{(n,m):m\in\mathbb N\}\cap U\neq\emptyset$, and it must contain a smallest element in such a set, and let
	$(n,x_n)=\min\{(n,m):m\in\mathbb N\}\cap U$.
	Then the family
	$\{L\setminus\da \{(n,x_n):n\geq k\}: k\geq M\}$ is a directed open cover of $L$, but no element includes $U$, a contradiction. Hence, $U=\emptyset$.
	It is trivial that $\emptyset\ll V$ for all $V\in\sigma(L)$.
	
	From (1), we deduce the following.
	
	(2) $\Sigma L$ is open well-filtered.
	
	(3) For each $n\in\mathbb N$, define $K_n=L\setminus\bigcup_{1\leq k\leq n}L_n$.
	Similarly, we can prove that $\Sigma K_n$ is open well-filtered.
	
	(4) The intersection $\bigcap_{n\in\mathbb N} \Sigma K_n=\max L$ which is homeomorphic to  $\mathbb N$ equipped with the cofinite topology, hence is not an open well-filtered space.
\end{example}

From the above example, we have that the intersection of open-well-filtered spaces need not be open well-filtered. Therefore, by Theorem \ref{th00}, we have the following result.
\begin{corollary}
	The category {\bf OWF}  is not reflective in ${\bf Top_0}$.
\end{corollary}

\section{Conclusion}

In this paper we proved that if a reflective  subcategory of ${\bf Top}_0$  contains a non $T_1$ space and satisfies the condition (K2) in the criteria proposed by Lawson and Keimel, then it also satisfies the rest of the conditions (K1), (K3) and (K4). Using this result , we  deduced that several subcategories are not reflective, and thus give the negative answers to some open problems. We expect that this result might be further used to check the reflectivity of other subcategories of ${\bf Top}_0$.

\end{document}